\documentclass{amsart}
\usepackage{amssymb}
\usepackage{braket}
\usepackage{mathrsfs}
\usepackage{ifthen}
\usepackage{here}
\usepackage{todonotes}
\usepackage{tikz}
\usetikzlibrary{patterns}
\usepackage{comment} 
\usepackage{mleftright}
\usepackage{hyperref}
\usepackage{mathrsfs}
\usepackage{cleveref}
 \usepackage[all]{xy}
 \usepackage{amscd}
 \usepackage{amsrefs}
 \usepackage{color}
 \usepackage{enumitem}
\newlist{steps}{enumerate}{1}
\setlist[steps, 1]{label = Step \arabic*:}
\usepackage[abs]{overpic}
\usepackage[latin1]{inputenc}
\usepackage{tikz-cd}
\usepackage{marginnote}

\makeatletter
\DeclareRobustCommand\widecheck[1]{{\mathpalette\@widecheck{#1}}}
\def\@widecheck#1#2{%
   \setbox\z@\hbox{\m@th$#1#2$}%
   \setbox\tw@\hbox{\m@th$#1%
      {%
         \vrule\@width\z@\@height\ht\z@
         \vrule\@height\z@\@width\wd\z@}$}%
   \dp\tw@-\ht\z@
   \@tempdima\ht\z@ \advance\@tempdima2\ht\tw@ \divide\@tempdima\thr@@
   \setbox\tw@\hbox{%
      \raise\@tempdima\hbox{\scalebox{1}[-1]{\lower\@tempdima\box\tw@}}}%
   {\ooalign{\box\tw@ \cr \box\z@}}}
\makeatother

\theoremstyle{plain}
\newtheorem{thm}{Theorem}[section]
\crefname{thm}{Theorem}{Theorems}
\Crefname{thm}{Theorem}{Theorems}

\crefname{prop}{Proposition}{Propositions}
\Crefname{prop}{Proposition}{Propositions}

\crefname{lem}{Lemma}{Lemmas}
\Crefname{lem}{Lemma}{Lemmas}
\newtheorem{cor}[thm]{Corollary}
\crefname{cor}{Corollary}{Corollaries}
\Crefname{cor}{Corollary}{Corollaries}

\crefname{claim}{Claim}{Claims}
\Crefname{claim}{Claim}{Claims}

\crefname{property}{Property}{Properties}
\Crefname{property}{Property}{Properties}

\crefname{problem}{Problem}{Problems}
\Crefname{problem}{Problem}{Problems}

\crefname{conjecture}{Conjecture}{Conjecture}
\Crefname{conjecture}{Conjecture}{Conjecture}

\theoremstyle{definition}

\crefname{defn}{Definition}{Definitions}
\Crefname{defn}{Definition}{Definitions}

\crefname{notation}{Notation}{Notations}
\Crefname{notation}{Notation}{Notations}

\crefname{convention}{Convention}{Conventions}
\Crefname{convention}{Convention}{Conventions}

\crefname{cond}{Condition}{Conditions}
\Crefname{cond}{Condition}{Conditions}

\crefname{assum}{Assumption}{Assumptions}
\Crefname{assum}{Assumption}{Assumptions}

\crefname{conj}{Conjecture}{Conjectures}
\Crefname{conj}{Conjecture}{Conjectures}

\crefname{claim1}{Claim}{Claims}
\Crefname{claim1}{Claim}{Claims}

\Crefname{ques}{Question}{Question}
\crefname{ques}{Question}{Question}

\theoremstyle{remark}
\newtheorem{rem}[thm]{Remark}
\crefname{rem}{Remark}{Remarks}
\Crefname{rem}{Remark}{Remarks}
\newtheorem{ex}[thm]{Example}
\crefname{ex}{Example}{Examples}
\Crefname{ex}{Example}{Examples}

\crefname{section}{Section}{Sections}
\Crefname{section}{Section}{Sections}
\crefname{subsection}{Subsection}{Subsections}
\Crefname{subsection}{Subsection}{Subsections}
\crefname{figure}{Figure}{Figures}
\Crefname{figure}{Figure}{Figures}

\newtheorem*{acknowledgement}{Acknowledgement}

\newcommand{\Z}{\mathbb{Z}}

\newcommand{\Diff}{\mathrm{Diff}}
\newcommand{\Homeo}{\mathrm{Homeo}}

\newcommand{\del}{\partial}

\newcommand{\rank}{\mathop{\mathrm{rank}}\nolimits}

\newcommand{\id}{\mathrm{id}}

\def\ker{\operatorname{Ker}}

\def\rank{\operatorname{rank}}

\def\id{\operatorname{id}}

\newcommand{\mbar}[1]{{\ooalign{\hfil#1\hfil\crcr\raise.167ex\hbox{--}}}}

     \RequirePackage{rotating}                   
    \def\HMt{%
       \setbox0=\hbox{$\widehat{\mathit{HM}}$}
       \setbox1=\hbox{$\mathit{HM}$}
       \dimen0=1.1\ht0
       \advance\dimen0 by 1.17\ht1
       \smash{\mskip2mu\raise\dimen0\rlap{%
          \begin{turn}{180}
              {$\widehat{\phantom{\mathit{HM}}}$}
           \end{turn}} \mskip-2mu    
                \mathit{HM}
                    }{\vphantom{\widehat{\mathit{HM}}}}{}}

\title{On absolutely exotic diffeomorphisms of 4-manifolds}

\author{Hokuto Konno}
\address{Graduate School of Mathematical Sciences, the University of Tokyo, 3-8-1 Komaba, Meguro, Tokyo 153-8914, Japan}
\email{konno@ms.u-tokyo.ac.jp}

\author{Abhishek Mallick}
\address{Department of Mathematics, Dartmouth College, 29 N Main St, Hanover, New Hampshire, 03755, USA}
\email{abhishek.mallick@dartmouth.edu}

\author{Masaki Taniguchi} 
\address{Department of Mathematics, Graduate School of Science, Kyoto University, Kitashirakawa Oiwake-cho, Sakyo-ku, Kyoto 606-8502, Japan}
\email{taniguchi.masaki.7m@kyoto-u.ac.jp}

\begin{document}

\begin{abstract}
We prove that there exist infinitely many contractible compact smooth $4$-manifolds $C$ that admit absolutely exotic diffeomorphisms of infinite order in $\pi_0(\Diff(C))$.
By ``absolutely", we mean that isotopies are not required to be relative to the boundary.
This follows from a theorem that produces absolutely exotic diffeomorphisms from relatively exotic diffeomorphisms, analogous to a theorem of Akbulut and Ruberman that produces absolutely exotic 4-manifolds from relatively exotic 4-manifolds.
\end{abstract}

\maketitle

\section{Main results}
\label{section Introduction}


Given a smooth manifold $W$ with boundary, we denote by $\Diff(W)$ the group of diffeomorphisms of $W$.
It has a natural subgroup $\Diff_{\del}(W)$, consisting of diffeomorphisms that are equal to the identity in a neighborhood of $\del W$.
Similarly, we define the absolute and relative homeomorphism groups, denoted by $\Homeo(W)$ and $\Homeo_{\del}(W)$.
We say that $f \in \Diff(W)$ is an \emph{absolutely exotic diffeomorphism} if $f$ represents a non-trivial element of 
\[
\ker\bigl(\pi_0(\Diff(W)) \to \pi_0(\Homeo(W))\bigr),
\]
and that $f \in \Diff_\del(W)$ is a \emph{relatively exotic diffeomorphism} if $f$ represents a non-trivial element of 
\[
\ker\bigl(\pi_0(\Diff_{\del}(W)) \to \pi_0(\Homeo_{\del}(W))\bigr).
\]

Recently, relatively exotic diffeomorphisms of contractible compact $4$-manifolds have been actively studied using families gauge theory \cites{konno-mallick-taniguchi-Dehn-twist,krushkal2024corksexoticdiffeomorphisms,konno2024localizinggroupsexoticdiffeomorphisms,KLMME1,KangParkTaniguchi-Dehn-twist,shivkumar2025exoticmathbbr4scompactlysupported}, most of which can be easily seen to be not absolutely exotic.
In addition, employing a different technique, Gabai--Gay--Hartman \cite{gabai2025pseudoisotopydiffeomorphisms4spherei} announced that $D^4$ admits a relatively exotic diffeomorphism of order~$2$ in the mapping class group, which would provide the first confirmed example of an absolutely exotic diffeomorphism of a contractible $4$-manifold (see Remark~\ref{rem:GGH}).

In this short note, we prove that there exist many contractible compact $4$-manifolds admitting absolutely exotic diffeomorphisms of infinite order:

\begin{thm}
\label{thm contractible}
There exist infinitely many contractible compact smooth $4$-manifolds $C$ that admit absolutely exotic diffeomorphisms of infinite order in $\pi_0(\Diff(C))$.
\end{thm}

More precisely, we construct infinitely many such mutually non-diffeomorphic manifolds $C$, distinguished by the Heegaard Floer homologies of their boundaries.
The boundaries of those $C$ we shall give are Haken, hence irreducible.

\cref{thm contractible} follows from the following general result, which produces absolutely exotic diffeomorphisms from relatively exotic diffeomorphisms, analogous to a theorem of Akbulut--Ruberman \cite{AR16} that constructs absolutely exotic $4$-manifolds from relatively exotic $4$-manifolds:

\begin{thm}
\label{thm upgrade}
Let $W$ be a compact smooth $4$-manifold with boundary.
Then $W$ admits a smooth codimension-0 embedding into a compact smooth $4$-manifold $V$ with boundary, homotopy equivalent to $W$, with the following property:
if $f \in \Diff_{\del}(W)$ is a relatively exotic diffeomorphism of $W$, then its extension to $V$ by the identity is an absolutely exotic diffeomorphism of $V$.
\end{thm}

In \cite{krushkal2024corksexoticdiffeomorphisms,KLMME1,KangParkTaniguchi-Dehn-twist}, it was proven that certain contractible compact $4$-manifolds $W$ admit a relatively exotic diffeomorphism $f$ of infinite order in $\pi_0(\Diff_\del(W))$.
Applying \cref{thm upgrade} to powers of $f$ then shows that there exists a contractible $4$-manifold $C$ admitting an absolutely exotic diffeomorphism of infinite order in $\pi_0(\Diff(C))$.
We shall see how to obtain infinitely many such $C$.

\begin{rem}
Absolutely exotic diffeomorphisms of other $4$-manifolds with boundary have already been studied in \cite[Theorem~1.4]{IKMT22} using the families Kronheimer--Mrowka invariant, which differs from the technique employed in this paper. The smallest example that can be obtained from \cite[Theorem~1.4]{IKMT22} has $b_2 = 4$.
\end{rem}

Using (the proof of) \cref{thm upgrade}, we can easily upgrade results on relatively exotic diffeomorphisms to results on absolutely exotic diffeomorphisms. 
For example, from a theorem of Lin \cite{JL20}, we have:

\begin{cor}
\label{cor: stabilization}
There exists a simply-connected, compact, smooth $4$-manifold with boundary that admits an absolutely exotic diffeomorphism which survives after taking the connected sum with a single copy of $S^2 \times S^2$.
\end{cor}

We shall also discuss higher homotopy groups of absolute diffeomorphism groups in \cref{Section: Higher homotopy groups}. 
\section{Akbulut--Ruberman cobordisms}

\cref{thm upgrade} follows from a key construction of homology cobordisms due to Akbulut and Ruberman~\cite{AR16}, which yields absolutely exotic 4-manifolds from relatively exotic 4-manifolds, such as corks.
What we need is a slight analysis of their construction, together with an input from $3$-dimensional topology, which can be summarized as follows; it provides a comparison result between the relative and absolute diffeomorphism groups.
To state this, let 
\[
i : \Diff_{\del}(W) \hookrightarrow \Diff(W)
\]
denote the natural inclusion map.
We say that a smooth oriented cobordism $X$ from $M$ to $N$ is \emph{invertible} if there exists a smooth oriented cobordism $X'$ from $N$ to $M$ such that $X \cup_{N} X'$ is diffeomorphic to $M \times [0,1]$.
We observe that the invertible homology cobordism constructed by Akbulut and Ruberman Akbulut--Ruberman~\cite[Section 2]{AR16}  possesses the following remarkable property:

\begin{thm}
\label{main theorem}
Let $W$ be a compact smooth $4$-manifold with boundary. 
Then there exists an invertible homology cobordism $X$ from $\del W$ to a $3$-manifold $N$ such that the $4$-manifold 
\[
V = W \cup_{\del W} X
\]
satisfies the following property:

\begin{itemize}
\item $V$ is homotopy equivalent to $W$.
\item $i_\ast : \pi_0(\Diff_\del(V)) \to 
\pi_0(\Diff(V))$ is injective.
\item $i_\ast : \pi_n(\Diff_\del(V)) \to \pi_n(\Diff(V))$ is isomorphic for $n>0$.
\end{itemize}
\end{thm}

\begin{proof}
Given $W$, Akbulut--Ruberman~\cite[Proof of Theorem A]{AR16} constructed a certain invertible homology cobordism $X$ from $M=\del W$ to a 3-manifold $N$ for which $V=W\cup_M X$ is homotopy equivalent to $W$.
(See \cite[Lemma 3.8]{yasui2025corksexotic4manifoldsgenus} for a detailed proof of the homotopy equivalence.)
As mentioned in \cite[Proof of Theorem~6.1]{AR16}, the $3$-manifold $N$ is obtained by gluing hyperbolic $3$-manifolds along incompressible tori, and hence is a Haken manifold.
Recall that a graph manifold is characterized as an irreducible $3$-manifold whose JSJ decomposition has no hyperbolic pieces.
Thus, $N$ is not a graph manifold, and in particular it is not a Seifert fibered $3$-manifold.
However, in general, a Haken manifold $Y$ that is not Seifert fibered satisfies that all higher homotopy groups of $\Diff(Y)$ vanish 
(see \cite[Theorem~2]{hatcher1999spacesincompressiblesurfaces}). Hence,
\begin{align}
\label{eq: higher vanish}
\pi_n(\Diff(N)) = 0 \quad \text{for all } n > 0.
\end{align}

On the other hand, the inclusion map $i : \Diff_{\del}(V) \hookrightarrow \Diff(V)$ fits into the exact sequence
\begin{align}
\label{eq: exact seq}
1 \longrightarrow \Diff_{\del}(V) \xrightarrow{i} \Diff(V) \xrightarrow{r} \Diff(N),
\end{align}
where $r : \Diff(V) \to \Diff(N)$ is the restriction map to the boundary, which is surjective onto a union of connected components.
This exact seqence gives rise to long exact sequences of homotopy groups, whereas the very final map on $\pi_0$ need not be surjective.
The remaining assertions of \cref{main theorem} follow from these long exact sequences together with \eqref{eq: higher vanish}.
\end{proof}

\section{From relative to absolute}

We now prove the results stated in \cref{section Introduction}.

\begin{proof}[Proof of \cref{thm upgrade}]
From \cref{main theorem}, we obtain an invertible homology cobordism $X$ from $M=\del W$ to some 3-manifold $N$ such that \[
i_\ast : \pi_0(\Diff_\del(V)) \to \pi_0(\Diff(V))
\]
is injective for
$V = W \cup_{\del W} X$.
By the invertibility of $X$, there exists a cobordism $X'$ from $N$ to $M$ such that $X \cup_{N} X'$ is diffeomorphic to $M \times [0,1]$.
Fix a diffeomorphism 
\[
X \cup_{N} X' \cong M \times [0,1],
\]
which induces a diffeomorphism
\begin{align}
\label{eq: identification}
V \cup_N X' = W \cup_M X \cup_N X' \cong W
\end{align}

Via the diffeomorphism \eqref{eq: identification}, the diffeomorphism
\[
f \cup \id_{X} \cup \id_{X'} \in \Diff_\del(W \cup_M X \cup_N X')
\]
corresponds to $f \in \Diff_\del(W)$.
Hence, the nontriviality of the relative mapping class 
$[f] \in \pi_0(\Diff_\del(W))$ implies that 
$[f \cup \id_{X}]$ is nontrivial in $\pi_0(\Diff_\del(V))$.
Since 
$i_\ast : \pi_0(\Diff_\del(V)) \to \pi_0(\Diff(V))$ 
is injective by \cref{main theorem}, 
it follows that $[f \cup \id_{X}]$ is nontrivial in 
$\pi_0(\Diff(V))$.

On the other hand, since the topological mapping class $[f]_{\mathrm{top}}$ is trivial in 
$\pi_0(\Homeo_\del(W))$,  
$[f \cup \id_{X}]_{\mathrm{top}}$ is trivial in 
$\pi_0(\Homeo(V))$.
Therefore, $f \cup \id_{X}$ is an absolutely exotic 
diffeomorphism of $V$.
\end{proof}

\begin{proof}[Proof of \cref{thm contractible}]
The first ingredient is the theorem of Kang, Park, and the third author \cite{KangParkTaniguchi-Dehn-twist}, which states that for any contractible compact smooth $4$-manifold $W$ whose boundary is a Seifert-fibered $3$-manifold, the boundary Dehn twist $\tau \in \Diff_\partial(W)$ is a relatively exotic diffeomorphism and has infinite order in $\pi_0(\Diff_\partial(W))$.
Let $V$ be the $4$-manifold given by \cref{thm upgrade}, which is homotopy equivalent to $W$ and contains $W$ as an embedded submanifold.
Then $V$ is contractible, and we see that the extension of $\tau$ to $V$ is an absolutely exotic diffeomorphism of $V$ of infinite order in $\pi_0(\Diff(V))$ by applying \cref{thm upgrade} to powers of $\tau$.

Infinitely many examples of such $V$, distinguished by the Floer homologies of $\partial V$, can be obtained as follows.
Let $W_n$ be a sequence of contractible, compact smooth $4$-manifolds whose boundaries are Seifert-fibered $3$-manifolds, and suppose that the ranks of their reduced Heegaard Floer homologies satisfy
\begin{align}
\label{eq: rank goes off to the infinity}
\rank_{\mathbb{F}} H\!F^{\mathrm{red}}(\partial W_n) \to +\infty
\end{align}
as $n \to +\infty$, where we consider Floer homology with $\mathbb{F}=\mathbb{F}_2$-coefficients.
To be specific, consider the Casson--Harer family, a well-known sequence of Brieskorn homology $3$-spheres that bound contractible smooth $4$-manifolds.  
If we choose a certain subfamily $\{Y_n\}_n$ of the Casson--Harer family and take contractible smooth $4$-manifolds $W_n$ bounded by $Y_n$, the family $\{W_n\}_n$ satisfies the assumption \eqref{eq: rank goes off to the infinity}.
For instance, 
\[
Y_n=\Sigma(2n,2n-1,2n+1)
\]
is such an example (see \cite[Table~2]{akbulut-karakurt-Mazur-manifolds}) for $n \geq 2$. 

Let $V_n$ be the $4$-manifold, homotopy equivalent to $W_n$, given by \cref{thm upgrade}.
After passing to a subsequence if necessary, we see that the $V_n$ are pairwise distinct, and, as noted above, each $V_n$ admits an absolutely exotic diffeomorphism of infinite order.

To show that $\{V_n\}_n$ contains a pairwise distinct subfamily, it suffices to prove that
\begin{align}
\label{eq: rank diverges}
\rank_{\mathbb{F}} H\!F^{\mathrm{red}}(\partial V_n) \to  +\infty
\end{align}
as $n \to +\infty$.
This follows immediately from the inequality
\begin{align}
\label{eq: rank ineq}
\rank_{\mathbb{F}} H\!F^{\mathrm{red}}(\partial W_n) \leq \rank_{\mathbb{F}} H\!F^{\mathrm{red}}(\partial V_n)
\end{align}
together with \eqref{eq: rank goes off to the infinity}. The inequality \eqref{eq: rank ineq} is established as follows.

Recall from the proof of \cref{thm upgrade} that each $V_n$ is obtained as
\[
V_n = W_n \cup_{\partial W_n} X_n,
\]
where $X_n$ is an invertible homology cobordism for each $n$.
Let $X_n'$ be a smooth cobordism from $\del V_n$ to $\del W_n$ such that $X_n \cup_{\del V_n} X_n'$ is diffeomorphic to $\del W_n \times [0,1]$.
Then the cobordism map 
\[
H\!F^{+}(X_n \cup_{\del V_n} X_n') : H\!F^{+}(\del W_n) \to H\!F^{+}(\del W_n) 
\]
in the plus version of Heegaard Floer homology is the identity map. 
Hence, 
\[
H\!F^{+}(X_n) : H\!F^{+}(\del W_n) \to H\!F^{+}(\del V_n)
\]
is injective.
Since the reduced Floer homology is characterized as the $U$-torsion part of the plus Floer homology, and the cobordism map $H\!F^{+}(X_n)$ is $U$-equivariant, it follows that the restriction of $H\!F^{+}(X_n)$ defines a well-defined, injective homomorphism
\[
H\!F^{\mathrm{red}}(X_n) : H\!F^{\mathrm{red}}(\partial W_n) \to H\!F^{\mathrm{red}}(\partial V_n),
\]
which verifies \eqref{eq: rank ineq}.
\end{proof}

\begin{rem}\label{rem:GGH}
Gabai--Gay--Hartman \cite{gabai2025pseudoisotopydiffeomorphisms4spherei} announced that some diffeomorphism $f \in \Diff_\del(D^4)$ (of order at most~$2$ in the mapping class group) is non-trivial in $\pi_0(\Diff_\del(D^4))$, and hence relatively exotic by topological triviality via the Alexander trick.
We remark that the relative exotica of $f$ immediately implies that $f$ is an absolutely exotic diffeomorphism of $D^4$.
Indeed, an exact sequence analogous to \eqref{eq: exact seq} for $V=D^4$ induces an exact sequence
\[
\pi_1(\Diff(S^3)) \xrightarrow{\del_\ast} \pi_0(\Diff_\del(D^4))
\xrightarrow{i_\ast} \pi_0(\Diff(D^4)).
\]
The image of the generator of $\pi_1(\Diff(S^3)) \cong \Z/2$ under $\del_\ast$ is the boundary Dehn twist of $D^4$, which can be seen to be trivial in $\pi_0(\Diff_\del(D^4))$ by using the standard circle action on $D^4$.
Hence $i_\ast$ is injective, so the diffeomorphism $f$ is non-trivial in $\pi_0(\Diff(D^4))$.

In contrast, most other previously known relatively exotic diffeomorphisms of compact contractible $4$-manifolds $C$ \cites{konno-mallick-taniguchi-Dehn-twist,KLMME1,KangParkTaniguchi-Dehn-twist} are boundary Dehn twists of $4$-manifolds with Seifert boundaries, and are therefore obviously trivial in $\pi_0(\Diff(C))$.
\end{rem}

It is worth further examining the Gabai-Gay-Hartman example \cite{gabai2025pseudoisotopydiffeomorphisms4spherei}. 
If we apply the proof of \cref{thm upgrade} to their example $f \in \Diff_\partial(D^4)$ discussed in Remark~\ref{rem:GGH}, we obtain:

\begin{rem}
Let $W$ be a compact, oriented, smooth $4$-manifold such that there exists a compact, oriented, smooth $4$-manifold $W'$ with $\partial W \cong -\partial W'$ and $W \cup_{\partial W} W'$ is diffeomorphic to $S^4$. 
Then $W$ admits a relatively exotic diffeomorphism of order 2 in $\pi_0(\Diff_\del(W))$. 
Moreover, if $\pi_1(\Diff(\partial W)) = 0$, then $W$ admits an absolutely exotic diffeomorphism of order 2 in $\pi_0(\Diff_\del(W))$.

To see this, fix a smooth embedding $D^4 \hookrightarrow W$ away from $\partial W$. 
Via this embedding, consider the extension to $W$ of the diffeomorphism $f \in \Diff_\partial(D^4)$ that is announced to be non-trivial in \cite{gabai2025pseudoisotopydiffeomorphisms4spherei}.
We can further extend $f$ to $W \cup_{\partial W} W'$. Since $W \cup_{\partial W} W'$ is diffeomorphic to $S^4$, the non-triviality of $f$ in $\pi_0(\Diff_\partial(D^4))$ (and hence in $\pi_0(\Diff(S^4))$) implies that $f$ is non-trivial in $\pi_0(\Diff_\partial(W))$, yielding an order-2 relatively exotic diffeomorphism. 
Moreover, if $\pi_1(\Diff(\partial W)) = 0$, then, as in the proof of \cref{thm upgrade}, the map $\pi_1(\Diff_\partial(W)) \to \pi_1(\Diff(W))$ is injective, so that $f$ is an order-2 absolutely exotic diffeomorphism of $W$.
\end{rem}

\begin{proof}[Proof of Corollary~\ref{cor: stabilization}]
Let $W$ be the punctured $K3$ surface $K3 \setminus \operatorname{int} D^4$. 
Lin \cite{JL20} proved that the boundary Dehn twist of $W\#S^2\times S^2$ is a relatively exotic diffeomorphism.
The proof of \cref{thm upgrade} provides an invertible homology cobordism $X$ from $S^3$ to a homology 3-sphere $N$ with $\pi_1(\Diff(N))=0$.
As in the proof of \cref{thm upgrade}, the invertibility of $X$ and that$\pi_1(\Diff(N)) = 0$ imply that the extension of the Dehn twist to $V = W \cup_{\partial W} X$ is an absolutely exotic diffeomorphism of $V$, following from the fact that the boundary Dehn twist on $W$ is relatively exotic.
The same argument with the relative exotica of the boundary Dehn twist on $W\#S^2\times S^2$ shows that the extension of the Dehn twist to $V\#S^2\times S^2 = (W\#S^2\times S^2) \cup_{\partial W} X$ is an absolutely exotic diffeomorphism. This completes the proof.
\end{proof}

Note that any relative mapping class on simply connected 4-manifolds with connected boundary with $b_1=0$ is known to be killed by finitely many stabilizations \cite{OrsonPowell2022}. Thus the above diffeomorphism is also killed by finitely many stabilizations.

\section{Higher homotopy groups}\label{Section: Higher homotopy groups}

So far, we have focused on $\pi_0(\Diff)$, but this can be easily generalized to $\pi_n(\Diff)$ for any $n \ge 0$.
Given a smooth manifold $W$ with boundary, we say that a homotopy class $\gamma \in \pi_n(\Diff(W))$ is \emph{absolutely exotic} if $\gamma$ is a non-trivial element of 
\[
\ker\bigl(\pi_n(\Diff(W)) \to \pi_n(\Homeo(W))\bigr),
\]
and that $\gamma \in \pi_n(\Diff_\del(W))$ is \emph{relatively exotic} if $\gamma$ is a non-trivial element of 
\[
\ker\bigl(\pi_n(\Diff_{\del}(W)) \to \pi_n(\Homeo_{\del}(W))\bigr).
\]
If $W$ is embedded into a smooth manifold $V$ as a codimension-$0$ submanifold, then a family of diffeomorphisms $\{f_b\}_{b \in S^{n}} \subset \Diff_\del(W)$ representing $\gamma$ can be extended to $V$ by the identity outside $W$, yielding an element of $\pi_n(\Diff(V))$.
With this preparation, \cref{thm upgrade} can be generalized as follows:

\begin{thm}
\label{thm upgrade higher}
Let $W$ be a compact smooth $4$-manifold with boundary and let $n \geq 0$.
Then $W$ admits a smooth codimension-0 embedding into a compact smooth $4$-manifold $V$ with boundary, homotopy equivalent to $W$, with the following property:
if a homotopy class $\gamma \in \pi_n(\Diff_{\del}(W))$ is relatively exotic, then its extension to $V$ by the identity lying in $\pi_n(\Diff(V))$ is absolutely exotic. 
\end{thm}

\begin{proof}
This follows from the proof of \cref{thm upgrade}, with $f$, $\pi_0(\Diff)$, and $\pi_0(\Homeo)$ replaced by $\{f_b\}_{b \in S^n}$, $\pi_n(\Diff)$, and $\pi_n(\Homeo)$, respectively, where $\{f_b\}_{b \in S^n}$ is a family of diffeomorphisms representing $\gamma$.
\end{proof}

Using \cref{thm upgrade higher}, we can produce many examples of $4$-manifolds that support absolutely exotic homotopy classes in $\pi_n(\Diff)$ while remaining rather small compared with $n$.

\begin{ex}
\label{ex:AucklyRuberman}
Given $n \ge 0$, Auckly--Ruberman~\cite{auckly2025familiesdiffeomorphismsembeddingspositive} constructed an exotic homotopy class $\gamma \in \pi_n(\Diff(Z))$ for some closed $4$-manifold $Z$. 
For example, one may take
\[
Z = K3 \# (n+1)(S^2 \times S^2),
\]
and their construction arises from the fact that the $K3$ surface and its logarithmic transforms are exotic.
By their construction of $\gamma$, it is straightforward to see that $\gamma$ can be represented by a family of diffeomorphisms $\{f_b\}_{b \in S^n} \subset \Diff(Z)$ whose supports are contained in 
\[
W = N(2) \# (n+1)(S^2 \times S^2),
\]
where $N(2)$ denotes the Gompf nucleus.
Regard each $f_b$ as a relative diffeomorphism of $W$, so that $\gamma \in \pi_n(\Diff_\del(W))$. 
Then the non-triviality of $\gamma$ in $\pi_n(\Diff(Z))$ immediately implies that $\gamma$ is non-trivial in $\pi_n(\Diff_\del(W))$.
On the other hand, since the triviality of $\gamma$ in $\pi_n(\Homeo(Z))$ follows from the fact that $N(2)$ and its logarithmic transforms are homeomorphic, it follows that $\gamma$ is trivial in $\pi_n(\Homeo_\del(W))$.
Thus $\gamma \in \pi_n(\Diff_\del(W))$ is relatively exotic.

Let $V$ be the $4$-manifold given by \cref{thm upgrade higher}, which is homotopy equivalent to $W$ and contains $W$ as an embedded submanifold.
Then $\gamma$ extends to an element of $\pi_n(\Diff(V))$ that is absolutely exotic.
\end{ex}

\begin{rem}
Enterprising readers can further reduce the support of the aforementioned diffeomorphisms from $W$ to $C \# (n+1) S^2 \times S^2$, where $C$ is an embedded contractible manifold in $Z$, following the work in \cite{konno2024localizinggroupsexoticdiffeomorphisms} by the first and the second authors.
\end{rem}

\begin{rem}
For higher homotopy groups, another example can be obtained from Watanabe's result \cite{Wa18} on the rational higher homotopy group of the diffeomorphism group of $S^4$, without using our technique. 

 In \cite{Wa18}, he provided infinite order elements in $\pi_i (\operatorname{Diff}(D^4, \partial))$ for certain positive $i$'s. By using the fibration \eqref{eq: exact seq} applied for $D^4$
and the affirmative answer to the Smale conjecture
\[
\Diff^+(S^3) \cong SO(4)
\] by Hatcher \cite{hatcher1983proof}, 
since the higher homotopy groups of $SO(4)$ are finite, sufficiently high powers of Watanabe's elements in $\pi_i (\operatorname{Diff}_\partial(D^4))$ survive in $\pi_i (\operatorname{Diff}(D^4))$, which are trivial in $\pi_i (\operatorname{Homeo}_\partial(D^4))$. In particular, this shows that the kernel of the natural map
\[
\pi_i (\operatorname{Diff}(D^4)) \to \pi_i (\operatorname{Homeo}(D^4))
\]
is not finite for certain positive $i$'s.
\end{rem}

\begin{acknowledgement}
We express our gratitude to Tetsuya Ito for answering some questions. 
HK was partially supported by JSPS KAKENHI Grant Numbers 25K00908, 25H00586. MT was partially supported by JSPS KAKENHI Grant Number 22K13921. HK and MT also thank AIM SQuaRE for its support.
\end{acknowledgement}

\bibliographystyle{plain}
\bibliography{tex}
\end{document}